\documentclass[12pt]{amsart}
\usepackage{amsmath}
\usepackage{amssymb}
\usepackage{xcolor}
\usepackage{mathdots}
\usepackage[autostyle]{csquotes}
\usepackage{float}
\usepackage{tikz-cd} 
\usetikzlibrary{patterns}
\usetikzlibrary{matrix}
\usepackage{tikz}
\usepackage{amssymb,latexsym,amsmath,amsthm,amsfonts,graphicx}
\usepackage{verbatim}
\usepackage[autostyle]{csquotes}
\usepackage{mathrsfs}
\usepackage{float}
\usepackage{pst-node}
\usepackage{mathtools}
\usepackage{faktor}
\usepackage{hyperref}
\usepackage{orcidlink}
\usepackage{cancel}
\usepackage{faktor}
\usetikzlibrary{arrows.meta}
\usetikzlibrary {decorations.pathmorphing}

\usetikzlibrary{arrows,positioning,shapes,fit,calc}
\usetikzlibrary{decorations.pathreplacing}
\usepackage{pgfplots}

\tikzstyle{littledot}=[circle, fill, inner sep=.4pt,minimum size=.4pt]
\tikzstyle{vertex}=[circle, fill, inner sep=1pt]

\newtheorem{thm}{Theorem}[section]
\newtheorem{lem}[thm]{Lemma}

\newtheorem{prop}[thm]{Proposition}
\newtheorem{cor}[thm]{Corollary}

\newtheorem{defn}[thm]{Definition}

\newtheorem{lemdef}[thm]{Definition/Lemma}

\newcommand{\R}{{\mathbb R}}

\newcommand{\Q}{{\mathbb Q}}
\newcommand{\N}{{\mathbb N}}

\newcommand{\Z}{{\mathbb Z}}

\newcommand{\cf}{\operatorname{cf}}

\newcommand{\fin}{\sim_\textrm{F}}
\newcommand{\timesf}{\cdot_\textrm{F}}

\newcommand{\cc}{\boldsymbol{c}_\omega}
\newcommand{\ON}{\mathbf{ON}}

\newcommand{\ot}{\operatorname{o.\!t.}}

\newcommand{\timesc}{\cdot_\omega}

\title{The countable condensation on linear orders}

\author[J. Brown]{Jennifer Brown}
\address{Mathematics Department, Augustana University, Sioux Falls, SD, USA}
\email{jennifer.brown@augie.edu}

\author[R. Su\'arez]{Ricardo Su\'arez}
\address{Natural Science Division, Pepperdine University, Malibu, CA, USA}
\email{josericardo.suarez@pepperdine.edu}
\keywords{linear orders, arithmetic of linear orders, order-preserving maps,  condensations of linear orders}
\subjclass{06A05, 06A11, 03E05}

\begin{document}
\maketitle

\begin{abstract}
The countable condensation on a linear order $L$ is the equivalence relation $\sim_\omega$ defined by declaring $x \sim_\omega y$ when the set of points between $x$ and $y$ is countable. We characterize the linear orders $L$ that condense to $1$ under the countable condensation by constructing a linear order $U$ that is universal for the order types $L$ such that $L/\!\!\sim_\omega\, \cong 1$.
We define a multiplication operation $\cdot_\omega$ on the class of linear orders by setting $M \cdot_\omega L$ to be the order type of $(ML)/\!\!\sim_\omega$ (where $ML$ denotes the lexicographic product), and show that the right identities for $\cdot_\omega$ are exactly the uncountable suborders of $U$. The order types of these uncountable suborders of $U$ form a left regular band under $\cdot_\omega$, and the order types of all suborders of $U$ form a semigroup.
\end{abstract}

\section{Introduction}\label{s: introduction}

A condensation of a linear order $L$ is an equivalence relation on $L$ whose equivalence classes are intervals (that is, convex subsets of $L$). Condensations of various kinds -- in particular, the finite condensation $\fin$, where two elements of a linear order are declared equivalent exactly when there are only finitely many elements lying between them -- have been used extensively as a tool for analyzing the structure of linear orders (see, for example, \cite{DuMi} or \cite{ErGu}). In \cite{BrSu}, we defined a multiplication on the class of linear orders in terms of the finite condensation: for linear orders $L$ and $M$, $L \timesf M := \ot(\faktor{LM}{\fin})$, where $LM$ is the lexicographic product of $L$ and $M$. The set of order types of linear orders that are right identities for the operation $\timesf$ is $\{\omega, \omega^*, \zeta\}$ (where $\omega^*$ denotes the reverse order on the ordinal $\omega$ and $\zeta$ denotes the order type of the integers), and we showed in \cite{BrSu} that this set forms a left-regular band. Moreover, $\zeta$ is a universal order type for linear orders that condense to $1$ modulo the finite condensation: for suppose $\faktor{L}{\fin} \cong 1$. Then $L$ has only one equivalence class modulo the finite condensation; that is, for any $x, y \in L$, there are only finitely many points of $L$ between $x$ and $y$. It follows that $L$ is either a finite linear order or is isomorphic to $\N$, $\N^*$, or $\Z$, all of whose order types embed into $\zeta$.

In the present work, we are again interested in algebraic structures that arise when a multiplication on linear orders is defined modulo a condensation, in this case the \textit{countable condensation}. For any linear order $L$ and $x, y \in L$, say that $x$ is equivalent to $y$ modulo the countable condensation, and write $x \sim_\omega y$, when there are only countably many elements between $x$ and $y$. (We show in Lemma/Definition \ref{countable condensation} that $\sim_\omega$ is, in fact, a condensation.) This condensation is mentioned, though not by name, in Exercise 4.10 in \cite{Ro}, where it is used to show that any uncountable dense subordering $A$ of $\R$ contains an $\aleph_1$-dense set. We define a multiplication $\timesc$ on the class of linear orders by $L \timesc M := \ot(\faktor{LM}{\sim_\omega})$. As in the case of the multiplication modulo the finite condensation, the set of linear orders that act as right identities for multiplication modulo the countable condensation does form a left-regular band, as we show in Theorem \ref{big left-regular band}; but unlike in the case of the finite condensation, this set of right identities for $\timesc$ is  infinite. We characterize the set of right identities for $\timesc$ in Theorem \ref{TFAE for right identity mod cc}. In order to do this, we first describe a universal set $U$ for order types that, when modded-out by the countable condensation, are isomorphic to the trivial linear order $1$ (i.e.\ that have only one equivalence class under $\sim_\omega$). The linear order $U$ is constructed by inserting a copy of the rationals between each pair of consecutive elements in $\omega_1^* + \omega_1$ (where $\omega_1^*$ denotes the ordinal $\omega_1$ under the reverse ordering). We then show that the set of right identities for multiplication modulo the countable condensation consists exactly of the uncountable order types that embed into $U$.

\section{Preliminaries}\label{s: preliminaries}

A standard reference for linear orders is \cite{Ro}, and we mostly follow the notational conventions found there. A notable exception is that we write our multiplication of linear orders \textit{lexicographically}: for linear orders $L$ and $M$, $LM$ is the linear order having the order type of the lexicographic order on $L \times M$. In this we follow the convention in \cite{ErGu}. We work in ZFC. For basic notions of set theory, see \cite{Ku}.  Since we treat ordinals as linear orders, our (lexicographic) notation for products of ordinals is the reverse of the standard one. For example, $2 \omega$ in this paper means the linear order that results from replacing each element of $2$ by a copy of $\omega$, so that $2 \omega \cong \omega + \omega$. $L^*$ denotes the underlying set $L$ with the reverse order. An interval $I$ in a linear order $L$ is a convex subset of $L$: if $a, b \in I$ and $a < x < b$, then $x \in I$. $[\{a, b\}]$ denotes either $\{x \in L: a \leq x \leq b\}$ (if $a \leq b$) or $\{x \in L: b \leq x \leq a\}$ (if $b<a$).

The order type of a linear order $L$ is denoted $\ot(L)$. We use the convention from \cite{Ro} that a representative linear order is chosen from each order type, and that, in the case of well-orders, that representative is an ordinal.  We will occasionally conflate linear orders and order types in situations where it should not cause confusion. To distinguish between, e.g., the order type of the first uncountable ordinal and its cardinality, we write $\omega_1$ for the former and $\aleph_1$ for the latter.

A subset $X$ of a linear order $L$ is said to be \textit{cofinal} if it is unbounded above in $L$; that is, if for every $l \in L$, there is an $x \in X$ such that $x \geq l$. Similarly, $X \subseteq L$ is \textit{cointial} if it is unbounded below in $L$.
The \textit{cofinality} of $L$, denoted $\cf(L)$, is the least length of a strictly increasing, cofinal sequence in $L$. The \textit{coinitiality} of $L$, denoted $\cf^*(L)$, is the least length of a strictly decreasing, coinitial sequence in $L$.

If $(P, \leq)$ is any partial order and $r \in P$, then 
$P \uparrow r:=\{q \in P: q \geq r\}$, and $P \downarrow r := \{q \in P: q \leq r\}.$ $P \uparrow r$ is read ``$P$ above $r$'', and $P \downarrow r$ is read ``$P$ below $r$''. If $L$ is a linear order and if we can write $L=L'+T$ for some $L'$ and $T$, then $T$ is called a \textit{tail}, or final segment, of $L$. Note that in general $L$ might have many tails. Similarly, if we can write $L=H+L'$ for some $L'$ and $H$, then $H$ is called a \textit{head}, or initial segment, of $L$. 

\begin{defn}\label{d: condensation}
A \textbf{condensation} of a linear order $L$ is an equivalence relation $\sim$ on $L$ whose equivalence classes are intervals. The equivalence class of $x \in L$ is denoted $\boldsymbol{c}(x)$. $\faktor{L}{\sim}$ is then a linear order as well, ordered by declaring $\boldsymbol{c}(x) < \boldsymbol{c}(y)$ when $l<m$ for all $l \in \boldsymbol{c}(x)$ and $m \in \boldsymbol{c}(y)$.
\end{defn}

\begin{lemdef}\label{countable condensation}
For any linear order $L$ and elements $x, y \in L$, say that $x \sim_\omega y$ if there are only countably many elements between $x$ and $y$; that is, $|[\{x,y\}]|\leq\aleph_0$.
For any linear order $L$, the relation $\sim_\omega$ is a condensation, which we will refer to as the \textbf{countable condensation}.
\end{lemdef}

\begin{proof}
For reflexivity: Let $x \in L$. Then $|[\{x,x\}]| = |\{x\}|= 1 \leq \aleph_0$, so $x \sim_\omega x$.

Symmetry is immediate since $[\{x,y\}] = [\{y,x\}]$.

For transitivity: suppose $x \sim_\omega y$ and $y \sim_\omega z$. If any two of $x, y,$ and $z$ are equal, then clearly $x \sim_\omega z$, so suppose $x, y$, and $z$ are all distinct. \textit{Case 1:} $x < y < z$. Then $[x,z] = [x,y] \cup [y,z]$, and each of $[x,y]$ and $[y,z]$ is countable by assumption, so $[x,z]$ is countable; so $x \sim_\omega z$. \textit{Case 2:} $x < z < y$. In this case $[x,z]$ is a subset of $[x,y]$, and we have supposed $[x,y]$ is countable, so $[x,z]$ is countable as well. \textit{Case 3:} $y < x < z$. Then $[z,x]$ is a subset of  $[y,x]$, which is countable, so $[z,x]$ is countable. \textit{Case 4:} $y < z < x$. Then $[z,x]$ is a subset of $[y,x]$, which is countable. \textit{Case 5:} $z < x < y$. Then $[z,x]$ is a subset of $[z,y]$, which is countable. \textit{Case 6:} $z < y < x$. Similarly to Case 1, we have that $[z,x]$ is the union of two sets, $[z,y]$ and $[y,x]$, which are by assumption countable. 

Thus $\sim_\omega$ is an equivalence relation. Next, we show that all of its equivalence classes are intervals. For $x \in L$, denote by $\cc(x)$ the equivalence class of $x$ modulo $\sim_\omega$:
\begin{align*}
\cc(x) & := \{y \in L: x \sim_\omega y\} \\
& = \{y \in L: \textrm{ there are only countably many points of $L$ between $x$ and $y$}\}.
\end{align*}
Suppose $y, z \in \cc(x)$ with $y < z$, and suppose $w \in L$ with $y < w < z$. If $w=x$, we are done, so suppose $w \neq x$. If $x<y$, then $[x,w] \subseteq [x,z]$, which is countable. If $y<x<z$, then $[\{w,x\}] \subseteq [y,z] = [y,x] \cup [x,z]$, which is a union of two countable sets. If $x>z$, then $[w,x] \subseteq [y,x]$, which is countable by assumption. In all cases, $w \sim_\omega x$, so that $w \in \cc(x)$. Thus $\cc(x)$ is an interval.
\end{proof}

We conclude this section by showing that the countable condensation of a well-ordered set is well-ordered; in particular, we have the class map $\alpha \mapsto \ot(\faktor{\alpha}{\sim_\omega})$ on the class of ordinals $\ON$.

\begin{lem}\label{well ordered-ness inherited for cc}
If $L$ is well-ordered, then so is $\faktor{L}{\sim_\omega}$.
\end{lem}

\begin{proof}
Suppose $L$ is a linear order such that $\faktor{L}{\sim_\omega}$ is not well-ordered. Then there are $x_n \in L$, for $n \in \omega$, such that $\boldsymbol{c}_\omega(x_0) > \boldsymbol{c}_\omega(x_1) > \cdots > \boldsymbol{c}_\omega(x_n) > \boldsymbol{c}_\omega(x_{n+1}) > \cdots$ in $\faktor{L}{\sim_\omega}$. But then also $x_0 > x_1 > \cdots > x_n > x_{n+1} > \cdots$ in $L$, so that $L$ is not well-ordered.
\end{proof}

In fact, the map $\alpha \mapsto \ot(\faktor{\alpha}{\sim_\omega})$ is an endomorphism of $\ON$, in the sense of a weakly order-preserving class map on $\ON$.

\begin{prop}\label{cc is an endomorphism on ON}
Let $\alpha$ and $\beta$ be ordinals with $\alpha < \beta$. Then $\ot(\faktor{\alpha}{\sim_\omega}) \leq \ot(\faktor{\beta}{\sim_\omega})$. 
\end{prop}

\begin{proof}
We know by Lemma \ref{well ordered-ness inherited for cc} that the countable condensation of a well-ordering is also well-ordered. Suppose $\alpha$ and $\beta$ are ordinals with $\alpha < \beta$.
Let $i$ denote the inclusion map from $\alpha$ to $\beta$, so that $i(x)$, for $x \in \alpha$, is its copy in $\beta$, and $i[\alpha]$ is the isomorphic copy of $\alpha$ inside of $\beta$.
Note that the equivalence class (modulo $\sim_\omega$) in $\alpha$ of an element $x \in \alpha$ might look different from the equivalence class in $\beta$ of $i(x)$. For this reason, for each $x \in \alpha$, denote $^\alpha \cc(x) = \{y \in \alpha: y \sim_\omega x\}$; and for each $x' \in \beta$, denote $^\beta \cc(x') = \{y \in \beta: y \sim_\omega x'\}$. Let $x \in \alpha$. If $^\beta \cc(i(x))$ is entirely contained in $i[\alpha]$, then $^\alpha \cc(x) \cong\, ^\beta \cc(i(x))$. If $^\beta \cc(i(x)) \not\subseteq i[\alpha]$ -- that is, if $^\beta \cc(i(x))$ intersects both $i[\alpha]$ and $\beta \setminus i[\alpha]$ -- then $^\alpha \cc(x) \cong\, ^\beta \cc(i(x)) \cap i[\alpha]$, since $\alpha$ is an initial segment of $\beta$. Thus $^\alpha \cc(x) \cong\, ^\beta \cc(i(x)) \cap i[\alpha]$ in any case.

Define a map $g:\faktor{\alpha}{\sim_\omega} \to \faktor{\beta}{\sim_\omega}$ by $g(^\alpha \cc(x)) =\, ^\beta \cc(i(x))$, for $x \in \alpha$. Then $g$ is injective: for if $^\beta \cc(i(x))=\,^\beta \cc(i(y))$ for some $x, y \in \alpha$, then $^\beta \cc(i(x)) \cap i[\alpha]=\,^\beta \cc(i(y)) \cap i[\alpha]$, so $^\alpha \cc(x) = \, ^\alpha \cc(y)$. (Observe that the map $g$ sends each interval $^\alpha \cc(x)$ in $\faktor{\alpha}{\sim_\omega}$ to its exact copy in $\beta$, except possibly for the very last one. That very last interval, say $^\alpha \cc(x)$ for some $x \in \alpha$,  might get sent by $g$ to a set containing some elements of $\beta \setminus i[\alpha]$.) Moreover, $g$ is order-preserving: for suppose $^\alpha \cc(x) <\, ^\alpha \cc(x')$ in $\faktor{\alpha}{\sim_\omega}$. First consider the case where $^\beta \cc(i(x)),\, ^\beta \cc(i(x')) \subseteq i[\alpha]$. Then $^\alpha \cc(x) \cong\, ^\beta \cc(i(x))$ and $^\alpha \cc(x') \cong\, ^\beta \cc(i(x'))$, so $g(^\alpha \cc(x)) =\, ^\beta \cc(i(x)) <\, ^\beta \cc(i(x')) = g(^\alpha \cc(x'))$. If it is not the case that both $^\beta \cc(x)$ and $^\beta \cc(x')$ are subsets of $i[\alpha]$, then necessarily $^\beta \cc(i(x)) \subseteq i[\alpha]$ and $^\beta \cc(i(x')) \not\subseteq i[\alpha]$. Observing that in this case $^\beta \cc(i(x')) \setminus i[\alpha] > i[\alpha]$, it again follows that $g(^\alpha \cc(x)) =\, ^\beta \cc(i(x)) <\, ^\beta \cc(i(x')) = g(^\alpha \cc (x'))$ in $\faktor{\beta}{\sim_\omega}$.

Therefore $g$ is an order-preserving injection from $\faktor{\alpha}{\sim_\omega}$ to $\faktor{\beta}{\sim_\omega}$, and since both of these sets are well-ordered by Lemma \ref{well ordered-ness inherited for cc}, the result follows.
\end{proof}

\section{The universal set for the countable condensation}\label{s: universal set for cc}

We would like to characterize the linear orders that condense to $1$ modulo the countable condensation. We immediately have that if $L$ is any countable linear order, then $\faktor{L}{\sim_\omega} \cong 1$ (for if $x, y \in L$, then $[\{x,y\}]$ is countable, as it is a subset of the countable set $L$, so $x \sim_\omega y$).  There are also many uncountable sets that reduce to $1$ modulo the countable condensation, though; for example, $\omega_1$ itself has this property.

\begin{lem}\label{omega_1 mod cc is 1}
$\faktor{\omega_1}{\sim_\omega} \cong 1$.
\end{lem}

\begin{proof}
Let $\alpha, \beta \in \omega_1$ with $\alpha < \beta$. As $\beta$ is countable, $[0, \beta]$ is countable. Since $[\alpha, \beta] \subseteq [0, \beta]$, the interval $[\alpha, \beta]$ is also countable, and so $\alpha \sim_\omega \beta$. Since any two elements of $\omega_1$ are in the same equivalence class modulo $\sim_\omega$, we have $\faktor{\omega_1}{\sim_\omega} \cong 1$.
\end{proof}

If $\alpha$ is any ordinal with $\alpha > \omega_1$, then $\faktor{\alpha}{\sim_\omega}$ has at least two elements. For example, the last element of the ordinal $\omega_1 + 1$ is not equivalent, modulo the countable condensation, to any element less than it, so that $\faktor{(\omega_1 + 1)}{\sim_\omega} \cong 2$ by Lemma \ref{omega_1 mod cc is 1}. 

Taking the countable condensation of the second uncountable ordinal results in a linear order of the same order type.

\begin{prop}\label{omega2 condenses to itself mod cc}
$\faktor{\omega_2}{\sim_\omega} \cong \omega_2$.
\end{prop}

\begin{proof}
By Lemma \ref{well ordered-ness inherited for cc}, $\faktor{\omega_2}{\sim_\omega}$ is isomorphic to an ordinal. Consider the sequence $\langle \alpha \omega_1: \alpha < \omega_2 \rangle$. This sequence is cofinal in $\omega_2$, for if $\alpha < \omega_2$ then $\alpha \leq \alpha \omega_1 < \omega_2$. For any $\alpha < \beta < \omega_2$, $|[\alpha \omega_1, \beta \omega_1]| = \omega_1$, so $\alpha \omega_1 \not\sim_\omega \beta \omega_1$. 
Thus the sequence $\langle \alpha \omega_1: \alpha < \omega_2 \rangle$ is strictly increasing, and so it has length $\omega_2$ as $\omega_2$ is regular. This means that there are at least $\aleph_2$-many equivalence classes modulo $\sim_\omega$ in $\omega_2$, namely $\{\boldsymbol{c}_\omega(\alpha \omega_1): \alpha < \omega_2\}$, so $\left|\faktor{\omega_2}{\sim_\omega}\right| \geq |\aleph_2|$. Since $\left|\faktor{\omega_2}{\sim_\omega}\right| \leq |\omega_2|$, we have $\left|\faktor{\omega_2}{\sim_\omega}\right| = \aleph_2$.

We have that if $\alpha < \beta < \omega_2$ then $\boldsymbol{c}_\omega(\alpha \omega_1) < \boldsymbol{c}_\omega(\beta \omega_1)$ in $\faktor{\omega_2}{\sim_\omega}$. That is, the classes $\{\boldsymbol{c}_\omega(\alpha \omega_1): \alpha < \omega_2\}$ form a suborder of $\faktor{\omega_2}{\sim_\omega}$ of type $\omega_2$. There is no suborder of $\faktor{\omega_2}{\sim_\omega}$ of order type $\omega_2 + 1$, for certainly any condensation of a linear order will be embeddable in that linear order. Thus the order type of $\faktor{\omega_2}{\sim_\omega}$ must be $\omega_2$.
\end{proof}

Note that Proposition \ref{omega2 condenses to itself mod cc} does not mean that the countable condensation does nothing to $\omega_2$; it simply means that performing this condensation on $\omega_2$ results in a linear order with the same order type. Convex $\omega$-sequences in $\omega_2$ do get condensed down to a point modulo the countable condensation. However, convex $\omega_1$-sequences in $\omega_2$ do not get condensed to a single point.

Recall the following theorem of Cantor, which states that the rationals are universal for countable linear orders.

\begin{thm}[Cantor]\label{rationals are universal}
Any countable linear order is isomorphic to a subordering of $\Q$. \hfill \qed  
\end{thm}

We next define a linear order $U$ that, as we will show in Theorem \ref{main result, universal set}, is universal for linear orders $L$ such that $\faktor{L}{\sim_\omega} \cong 1$. To see what form such a $U$ must take, consider a linear order $L$ that has only a single equivalence class modulo the countable condensation, and suppose for simplicity that $L$ has a least element. The intuition here is that $L$ must have a cofinal ``spine'' of order type at most $\omega_1$ with the ``gaps'' filled-in with some sort of countable linear order. We will map this spine to the strictly increasing cofinal sequence $\{u_\alpha: \alpha < \omega_1\}$ of $U$ (see Definition \ref{d: alternate rational long line}), and -- by Theorem \ref{rationals are universal} -- we can map each of those countable linear orders in the ``gaps'' into the copy of $\Q$ that we will insert between $u_\alpha$ and $u_{\alpha+1}$ in $U$.

\begin{defn}\label{d: alternate rational long line}
Let $\{u_\alpha: \alpha < \omega_1\}$ have order type $\omega_1$, and let $\{-u_\alpha: \alpha < \omega_1\}$ have order type $\omega_1^*$. For $\alpha < \omega_1$, let $\Q(\alpha)$ and $\Q(-\alpha)$ be isomorphic copies of $\Q$; and let $\Q(\textrm{mid})$ be another isomorphic copy of $\Q$. We define a linear order on the point set 
\[
U := \{-u_\alpha: \alpha \in \omega_1\} \cup \{u_\alpha: \alpha \in \omega_1\} \cup \left( \bigcup_{\alpha \in \omega_1} \Q(\alpha) \right) \cup \left( \bigcup_{\alpha \in \omega_1} \Q(-\alpha) \right) \cup \Q(\textrm{mid})
\]
as follows: first, we declare $\{-u_\alpha: \alpha < \omega_1\} < \{u_\alpha: \alpha < \omega_1\}$. For $\alpha < \beta < \omega_1$, we declare $\Q(-\beta) < \Q(-\alpha) < \Q(\textrm{mid}) <  \Q(\alpha) < \Q(\beta)$. For $\alpha < \omega_1$, we declare $u_\alpha < \Q(\alpha) < u_{\alpha + 1}$ and $-u_{\alpha+1} < \Q(-u_\alpha) < -u_\alpha$. Finally, we declare $-u_0 < \Q(\textrm{mid}) < u_0.$ We call this linear order $U$ the \textbf{$\omega_1$-lengthened rational line.}
\end{defn}

\begin{prop}\label{U condenses to 1 under cc}
$\faktor{U}{\sim_\omega} \cong 1$.
\end{prop}

\begin{proof}
For convenience, we write $\omega_1^* = \{-\alpha: \alpha < \omega_1\}$. Let $x, y \in U$; we show that $[x,y]$ is countable. There are several cases.

First suppose $x, y \in \Q(\alpha)$ for some $\alpha < \omega_1$. Then $[x,y] \subseteq \Q(\alpha)$, which is countable. The same reasoning works if $x, y \in \Q(-\alpha)$ for some $\alpha$ or if $x, y \in \Q(mid)$.

Next suppose $x \in \Q(\alpha)$ and $y \in \Q(\beta)$ for some $\alpha <  \beta < \omega_1$. Then $|[x,y]| \leq |\Q||\beta| \leq \aleph_0 \aleph_0 = \aleph_0.$ The same reasoning works if $x \in \Q(-\alpha)$ and $y \in \Q(-\beta)$ for some $\alpha < \beta < \omega_1$.

Finally, suppose $x \in \Q(-\alpha)$ and $y \in \Q(\beta)$ for some $\alpha, \beta < \omega_1$. Then $|[x,y]| \leq 2 |\Q| \max\{\alpha, \beta\} \leq \aleph_0  \aleph_0 = \aleph_0.$

Thus in all cases $[x,y]$ is countable. Therefore $x \sim_\omega y$ for all $x, y \in U$, so that $\faktor{U}{\sim_\omega} \cong 1$.
\end{proof}

\begin{prop}\label{things in U condense to 1 under cc}
If $L$ is isomorphic to a suborder of $U$, then $\faktor{L}{\sim_\omega} \cong 1$.
\end{prop}

\begin{proof}
Suppose $L \subseteq U$ with the induced order. For $x, y \in L$, $[\{x,y\}]_L \subseteq [\{x,y\}]_U$. By Proposition \ref{U condenses to 1 under cc}, $|[\{x,y\}]_U| \leq \omega$, so also $|[\{x,y\}]_L| \leq \omega$, so  $x \sim_\omega y.$
\end{proof}

\begin{lem}\label{bound on monotone sequences}
If $L$ is any linear order such that $\faktor{L}{\sim_\omega} \cong 1$, then any strictly increasing (resp.\ decreasing) sequence in $L$ must be of order type at most $\omega_1$ (resp.\ $\omega_1^*)$.
\end{lem}

\begin{proof}
We argue by contrapositive. Let $L$ be a linear order, and suppose that there is a strictly increasing sequence $\langle x_\alpha: \alpha < \omega_1 + 1 \rangle$ in $L$ of order type $\omega_1+1$. Since there are $\omega_1$-many elements of $L$ between $x_0$ and $x_{\omega_1}$, we have $x_0 \not\sim_{\omega} x_{\omega_1}$, and so $\faktor{L}{\sim_\omega}$ has at least $2$ elements. (A similar argument shows that $\faktor{L}{\sim_\omega}$ has at least $2$ elements if there is a strictly decreasing sequence of type $(\omega_1 + 1)^*$ in $L$.)
\end{proof}

\begin{cor}\label{cf at most omega 1 when L condenses to  1}
If $L$ is a linear order such that $\faktor{L}{\sim_\omega} \cong 1$, then $\cf(L) \leq \omega_1$ and $\cf^*(L) \leq \omega_1$.
\end{cor}

\begin{proof}
If $\cf(L)$ were greater than $\omega_1$, then $L$ would have a strictly increasing sequence of size greater than $\omega_1$; but this is not possible, by Lemma \ref{bound on monotone sequences}. Thus $\cf(L) \leq \omega_1$. A similar argument shows that $\cf^*(L) \leq \omega_1$.
\end{proof}

\begin{prop}\label{if L condenses to 1 mod cc then L has size at most omega1}
If $L$ is any linear order such that $\faktor{L}{\sim_\omega} \cong 1$, then $|L| \leq \aleph_1$.
\end{prop}

\begin{proof}
Let $L$ be a linear order with $\faktor{L}{\sim_\omega} \cong 1$. 

\textit{Case 1:} If $L$ has both a first element $x$ and a last element $y$, then since $x \sim_\omega y$, we have $|L| = |[x,y]| \leq \aleph_0$.

\textit{Case 2:} Suppose $L$ has a first element $x_0$ but no last element. Let $\langle x_\alpha: \alpha < \cf(L) \rangle$ be cofinal in $L$. Then for each $l \in L$, there is an $\alpha < \cf(L)$ such that $x_0 \leq l < x_\alpha$, so $L=\bigcup_{\alpha < \cf(L)}[x_0, x_\alpha)$. Then, noting that $\faktor{L}{\sim_\omega} \cong 1$ implies that $x_0 \sim_\omega x_\alpha$ for each $\alpha < \cf(L)$, we have 
\[
|L| = \left| \bigcup_{\alpha < \cf(L)} [x_0, x_\alpha) \right| \leq \sum_{\alpha < \cf(L)}|[x_0, x_\alpha)| \leq \sum_{\alpha < \cf(L)} \aleph_0 \leq \sum_{\alpha < \omega_1} \aleph_0 = \aleph_1
\]
as $\cf(L) \leq \omega_1$ by Corollary \ref{cf at most omega 1 when L condenses to  1}.

\textit{Case 3:} Suppose $L$ has a last element $x_0$ but no first element. In this case, we can find a coinitial sequence $\langle x_\alpha: \alpha < \cf^*(L) \rangle$ in $L$ of length at most $\omega_1$, and we can write $L=\bigcup_{\alpha < \cf^*(L)}(x_\alpha, x_0]$. Then since $x_\alpha \sim_\omega x_0$ for each $\alpha < \cf^*(L)$, we again have 
\[
|L| = \left| \bigcup_{\alpha < \cf^*(L)} (x_\alpha, x_0] \right| \leq \sum_{\alpha < \cf^*(L)}|(x_\alpha, x_0]| \leq \sum_{\alpha < \cf^*(L)} \aleph_0 \leq \sum_{\alpha < \omega_1} \aleph_0 = \aleph_1.
\]

\textit{Case 4:} Suppose $L$ has neither a first element nor a last element. Choose any $x_0 \in L$. By Case 2, $|[x_0, +\infty)| \leq \aleph_1$; and by Case 3, $|(-\infty, x_0]| \leq \aleph_1$; so $|L| = |(-\infty, x_0] \cup [x_0, +\infty)| \leq \aleph_1 + \aleph_1 = \aleph_1$. 

Therefore in all cases we have $|L| \leq \aleph_1$.

\end{proof}

Note that the converse of Proposition \ref{if L condenses to 1 mod cc then L has size at most omega1} does not hold; for example, $\faktor{(\omega_1 + 1)}{\sim_\omega} \cong 2$ even though $|\omega_1 + 1| = \aleph_1$. 

\begin{lem}\label{uncountable suborders of U have cf at least omega1}
If $L$ is a linear order with $\faktor{L}{\sim_\omega} \cong 1$ and $L$ is uncountable, then $\cf(L)=\omega_1$ or $\cf^*(L)=\omega_1$.
\end{lem}

\begin{proof}
We argue by contrapositive. Suppose $\faktor{L}{\sim_\omega} \cong 1$, and suppose that $\cf(L) \neq \omega_1$ and $\cf^*(L) \neq \omega_1$. By Corollary \ref{cf at most omega 1 when L condenses to  1}, this means that $\cf(L)$ and $\cf^*(L)$ are both countable. 
There are several cases.

If $L$ has both a first element $x$ and a last element $y$, then $|L| = |[x,y]| \leq \aleph_0$ because $x \sim_\omega y$. 

Suppose $L$ has a first element $x_0$ but no last element, so that (since we are assuming that $\cf(L)$ is countable) we can find a strictly increasing cofinal sequence $\langle x_n: n < \omega \rangle$ in $L$. Then for each $l \in L$, there is an $n<\omega$ such that $l \in [x_0, x_n)$; and also, for each $n<\omega$, the interval $[x_0,x_n)$ is countable, as we are assuming that $\faktor{L}{\sim_\omega} \cong 1$. Therefore 
\[
|L|= \left| \bigcup_{n<\omega} [x_0, x_n) \right| \leq \sum_{n<\omega} |[x_0, x_n)| \leq \sum_{n<\omega}\aleph_0 = \aleph_0. 
\]

Next, suppose that $L$ has a last element but no first element. In this case we can find a countable strictly decreasing coinitial sequence $\langle y_n:n<\omega \rangle$ in $L$ and write $L=\bigcup_{n<\omega}(y_n,y_0]$. Then, since each of the intervals $(y_n,y_0]$ must be countable, we have
\[
|L|= \left| \bigcup_{n<\omega} (y_n, y_0] \right| \leq \sum_{n<\omega} |(y_n, y_0]| \leq \sum_{n<\omega}\aleph_0 = \aleph_0. 
\]

Finally, suppose $L$ has neither a first element nor a last element, so that $\cf(L)=\omega$ and $\cf^*(L)=\omega$. Let $\langle x_n:n<\omega \rangle$ be increasing and cofinal in $L$, and let $\langle y_n:n<\omega \rangle$ be decreasing and coinitial in $L$, and suppose $x_0=y_0$. (We may arrange this, adding an extra element to the beginning of a cofinal or coinitial sequence as necessary.) We then have 
\begin{align*}
|L| & = |(-\infty, x_0] \cup [x_0, +\infty)| = \left| \left( \bigcup_{n<\omega} (y_n,y_0]\right) \cup \left( \bigcup_{n<\omega} [x_0, x_n) \right) \right| \\
& \leq \left| \bigcup_{n<\omega} (y_n,y_0] \right| + \left| \bigcup_{n<\omega} [x_0, x_n) \right| \leq \left( \sum_{n<\omega} |(y_n, y_0]| \right) + \left( \sum_{n<\omega} |[x_0, x_n)| \right) \\
& \leq \left( \sum_{n<\omega} \aleph_0 \right) + \left( \sum_{n<\omega} \aleph_0 \right) = \aleph_0 + \aleph_0 = \aleph_0.
\end{align*}
Therefore in all cases, if $\faktor{L}{\sim_\omega} \cong 1$ and both $\cf(L)$ and $\cf^*(L)$ are countable, then $L$ is countable. 
\end{proof}

\begin{prop}\label{writing L as a disjoint union of intervals}
Suppose $L$ is a linear order with a first element such that $\faktor{L}{\sim_\omega} \cong 1$ and $\cf(L)=\omega_1$. Then $L$ embeds into a linear order $L'$ such that  $\faktor{L'}{\sim_\omega} \cong 1$, $\cf(L') = \omega_1$, $L'$ has a first element, and $L'$ has a strictly increasing cofinal sequence $\langle y_\alpha: \alpha < \omega_1 \rangle$ such that $L'=\bigcup_{\alpha < \omega_1}[y_\alpha, y_{\alpha + 1})$.
\end{prop}

\begin{proof}
Suppose $L$ has a first element $x_0$, and let $X:=\langle x_\alpha: \alpha < \omega_1 \rangle$ be strictly increasing and cofinal in $L$. We construct the desired linear order $L'$ and cofinal sequence $\langle y_\alpha: \alpha < \omega_1 \rangle$ by adding to $X$ the suprema in $L$ of sequences in $X$, if these exist, or -- if these suprema do not exist -- by adding new elements to $L$. We do this by defining sequences $\{s_\alpha: \alpha < \omega_1\}$ and $\{n_\alpha: \alpha < \omega_1\}$ as follows.

Let $\alpha < \omega_1$. If $\alpha$ is a successor ordinal, then set $s_\alpha = n_\alpha = x_0$. Next, suppose that $\alpha < \omega_1$ is a limit ordinal.  If $x_\alpha = \sup\{x_\beta: \beta < \alpha \}$ in $L$, then set $s_\alpha = n_\alpha = x_0$. If $\sup\{x_\beta: \beta < \alpha\}$ exists in $L$ but $\sup\{x_\beta: \beta < \alpha\} \neq x_\alpha$, then we set $s_\alpha = \sup\{x_\beta: \beta < \alpha\}$ and set $n_\alpha = x_0$.
If $\sup\{x_\beta: \beta < \alpha\}$ does not exist in $L$, then we set $s_\alpha = x_0$, and we add a new element $n_\alpha$ to $L$ with the property that $\{x_\beta: \beta < \alpha\} < n_\alpha < \{l \in L: l>x_\beta \textrm{ for all } \beta < \alpha\}$. That is, in this case we add a new element that will be the supremum of $\{x_\beta: \beta < \alpha\}$ in $L'$. 

Let $L'$ consist of $L$ together with all of the newly-added elements $n_\alpha$. Let $Y$ consist of $X$ together with the newly-added elements $s_\alpha$ and $n_\alpha$. That is,
\[
L' := L \cup \{n_\alpha: \alpha < \omega_1\} \quad \textrm{and} \quad Y:=X \cup \{s_\alpha: \alpha < \omega_1\} \cup \{n_\alpha: \alpha < \omega_1\}.
\]
We still have $\faktor{L'}{\sim_\omega} \cong 1$: for suppose $l, m \in L'$ with $l<m$. $X$ is still cofinal in $L'$, as each new element $n_\alpha \in L' \setminus L$ was less than $x_\alpha \in X$. Then we can choose $\alpha < \omega_1$ such that $m < x_\alpha$. 
Observe that between any two elements of $L$, at most countably many new elements $n_\alpha$ were added in forming $L'$. Then since $\faktor{L}{\sim_\omega} \cong 1$, we have  $|[x_0, x_\alpha)| = |([x_0, x_\alpha) \cap L) \cup ([x_0, x_\alpha) \setminus L)| \leq \aleph_0 + \aleph_0 = \aleph_0$. Then $|[l,m)|$ is also countable, as $[l, m) \subseteq [x_0, x_\alpha)$, and so $l \sim_\omega m$ in $L'$.

Next, observe that $Y$ is cofinal in $L'$, and also $Y$ is well-ordered of order type $\omega_1$. We can then index the elements of $Y$ as $Y=\langle y_\alpha: \alpha < \omega_1 \rangle$ (and we will have $y_0=x_0$). By construction, $Y$ has the property that for each limit $\alpha < \omega_1$, $y_\alpha = \sup \{y_\beta: \beta < \alpha\}.$

Finally, we claim that $L'=\bigcup_{\alpha < \omega_1}[y_\alpha, y_{\alpha + 1})$: for let $l \in L'$. If $l \in Y$, we are done, so suppose $l \not\in Y$. The set $\{y_\alpha \in Y: y_\alpha > l\}$ is nonempty since $\cf(L)=\cf(L')=\omega_1$ (so that $L'$ has no last element). Then since $Y$ is well-ordered, $\{y_\alpha \in Y: y_\alpha > l\}$ has a least element  $y_{\alpha'}$. Suppose  $\alpha'$ is a successor ordinal; say $\alpha' = \beta+1$. By the minimality of $y_{\alpha'}$, it must be that $y_\beta < l$; so $l \in [y_\beta, y_{\beta+1})$, as desired. We now claim that in fact $\alpha'$ must be a successor ordinal: for suppose not. By an earlier claim, $y_{\alpha'} = \sup\{y_\beta: \beta < \alpha'\}$. By the minimality of $\alpha'$, we have that for all $\beta < \alpha', y_\beta \leq l$. But then we would have 
\[
\sup\{y_\beta: \beta < \alpha'\} \leq l < y_{\alpha'} = \sup\{y_\beta: \beta < \alpha'\},
\]
which is a contradiction.
\end{proof}

\begin{prop}\label{L' embeds into U}
Suppose $L$ is a linear order with a first element such that $\faktor{L}{\sim_\omega} \cong 1$ and $\cf(L)=\omega_1$. Suppose further that $L$ has a strictly increasing, cofinal sequence $\langle x_\alpha: \alpha < \omega_1 \rangle$ such that $L=\bigcup_{\alpha < \omega_1}[x_\alpha, x_{\alpha + 1})$. Then $L$ is isomorphic to a suborder of $U$. 
\end{prop}

\begin{proof}
Let $L$ be a linear order with a first element $x_0$ such that $\faktor{L}{\sim_\omega} \cong 1$ and $\cf(L)=\omega_1$, and let $\langle x_\alpha: \alpha < \omega_1 \rangle$ be a strictly increasing, cofinal sequence such that $L$ can be written as $L=\bigcup_{\alpha < \omega_1}[x_\alpha, x_{\alpha + 1})$. Note that the intervals $[x_\alpha, x_{\alpha + 1})$, for $\alpha < \omega_1$, are pairwise disjoint. For each $\alpha < \omega_1$, the subset $[x_\alpha, x_{\alpha+1})$ of $L$ is countable, by our assumption that $\faktor{L}{\sim_\omega} \cong 1$. Then by Theorem \ref{rationals are universal}, there is an embedding $f_\alpha$ mapping $(x_\alpha, x_{\alpha+1})$ to the interval $\Q(\alpha)=(u_\alpha, u_{\alpha + 1})$ of $U$. Define a map $f:L \to U$ as follows: for $l \in L$, there is a unique $\alpha_l<\omega_1$ such that $l \in [x_{\alpha_l}, x_{\alpha_l+1})$. Set
\[
f(l) = \left\{
\begin{array}{ll}
u_{\alpha_l}, & \textrm{ if } l=x_{\alpha_l} \\
f_{\alpha_l}(l), & \textrm{ if not.}
\end{array}
\right.
\]
We claim that $f$ is an embedding of $L$ into $U$: for suppose $l, m \in L$ with $l<m$. First suppose that $l, m \in [x_\alpha, x_{\alpha+_1})$ for some $\alpha< \omega_1$, so that $\alpha_l = \alpha_m = \alpha$. If $x_\alpha < l$, then $f(l) = f_\alpha(l) < f_\alpha(m) = f(m)$ because $f_\alpha$ is an embedding. If $l=x_\alpha$, then $f(l) = u_\alpha < f_\alpha(m)=f(m)$ because $f_\alpha(m) \in (u_\alpha, u_{\alpha + 1})$.
Next, suppose that $l \in [x_\alpha, x_{\alpha+1})$ and $m \in [x_\beta, x_{\beta+1})$ for some $\alpha<\beta$. Then $f(l) \in [u_\alpha, u_{\alpha+1})$ and $f(m) \in [u_\beta, u_{\beta+1})$, and $[u_\alpha, u_{\alpha+1}) < [u_\beta, u_{\beta+1})$ by construction of $U$, so $f(l) < f(m)$. Thus $f$ is an order-preserving map from $L$ into $U$; that is, $L$ embeds into $U$.
\end{proof}

\begin{prop}\label{L embeds into U, case cf(L)=omega1 with first element}
Let $L$ be a linear order with a first element such that $\faktor{L}{\sim_\omega} \cong 1$ and $\cf(L)=\omega_1$. Then $L$ is isomorphic to a suborder of $U$.  
\end{prop}

\begin{proof}
Suppose $L$ is a linear order with a first element such that $\faktor{L}{\sim_\omega} \cong 1$ and $\cf(L)=\omega_1$. Then by Proposition \ref{writing L as a disjoint union of intervals}, $L$ embeds into a linear order $L'$ having a first element such that $\faktor{L'}{\sim_\omega} \cong 1$, $\cf(L')=\omega_1$, and there is a strictly increasing cofinal sequence $\langle y_\alpha: \alpha < \omega_1 \rangle$ in $L'$ such that $L'=\bigcup_{\alpha < \omega_1}[y_\alpha, y_{\alpha + 1}).$ Then by Proposition \ref{L' embeds into U}, $L'$ embeds into $U$. That is, we have $L \hookrightarrow L'$ and $L' \hookrightarrow U$, and therefore $L \hookrightarrow U$.
\end{proof}

Reversing the order, we obtain, \textit{mutatis mutandis}, the following versions of Propositions \ref{writing L as a disjoint union of intervals}, \ref{L' embeds into U}, and \ref{L embeds into U, case cf(L)=omega1 with first element}, respectively.

\begin{prop}\label{writing L as a disjoint union of intervals, reverse}
Suppose $L$ is a linear order with a last element such that $\faktor{L}{\sim_\omega} \cong 1$ and $\cf^*(L)=\omega_1$. Then $L$ embeds into a linear order $L'$ such that  $\faktor{L'}{\sim_\omega} \cong 1$, $\cf^*(L') = \omega_1$, $L'$ has a last element, and $L'$ has a strictly decreasing coinitial sequence $\langle y_\alpha: \alpha < \omega_1 \rangle$ such that $L'=\bigcup_{\alpha < \omega_1}(y_{\alpha+1}, y_\alpha]$. \hfill \qed
\end{prop}

\begin{prop}\label{L' embeds into U, reverse}
Suppose $L$ is a linear order with a last element such that $\faktor{L}{\sim_\omega} \cong 1$ and $\cf^*(L)=\omega_1$. Suppose further that $L$ has a strictly decreasing, coinitial sequence $\langle x_\alpha: \alpha < \omega_1 \rangle$ such that $L=\bigcup_{\alpha < \omega_1}(x_{\alpha+1}, x_\alpha]$. Then $L$ is isomorphic to a suborder of $U$. \hfill \qed
\end{prop}

\begin{prop}\label{L embeds into U, case cf*(L)=omega1 with last element}
Let $L$ be a linear order with a last element such that $\faktor{L}{\sim_\omega} \cong 1$ and $\cf^*(L)=\omega_1$. Then $L$ is isomorphic to a suborder of $U$.  \hfill \qed
\end{prop}

We can now prove that the $\omega_1$-lengthened rational line $U$ is universal for the linear orders that condense to $1$ modulo the countable condensation.

\begin{thm}\label{main result, universal set}
Let $L$ be any linear order. Then $\faktor{L}{\sim_\omega} \cong 1$ if and only if $L$ is isomorphic to a suborder of $U$.
\end{thm}

\begin{proof}
If $L$ is isomorphic to a suborder of $U$, then $\faktor{L}{\sim_\omega} \cong 1$ by Proposition \ref{things in U condense to 1 under cc}.

Now suppose that $\faktor{L}{\sim_\omega} \cong 1$.

\textit{Case 1:} If $L$ is countable, then $L$ embeds into $\Q(1)$ by Theorem \ref{rationals are universal}. 

\textit{Case 2:} Suppose $L$ is uncountable and has a first element. Then $\cf(L)=\omega_1$ by Lemma \ref{uncountable suborders of U have cf at least omega1}, so $L$ embeds into $U$ by Proposition \ref{L embeds into U, case cf(L)=omega1 with first element}.

\textit{Case 3:} Next, suppose $L$ is uncountable and has a last element. Then $\cf^*(L)=\omega_1$ by Lemma \ref{uncountable suborders of U have cf at least omega1}, so $L$ embeds into $U$ by Proposition \ref{L embeds into U, case cf*(L)=omega1 with last element}.

\textit{Case 4:} Finally, suppose $L$ is uncountable with neither a first element nor a last element. Fix any $x_0 \in L$ and denote $L^-=(-\infty, x_0]$ and $L^+ = [x_0, \infty)$. Then we can embed $L^-$ into $U^-=(-\infty,-u_0]$ by cases 1 or 3, and we can embed $L^+$ into $U^+=[u_0,\infty)$ by cases 1 or 2.
\end{proof}

\section{The universal set $U$ and multiplication modulo the countable condensation}\label{s: The universal set U and multiplication modulo the countable condensation}

In this section, we define a multiplication operation modulo the countable condensation, and show how it relates to the universal set $U$ defined in Section \ref{s: universal set for cc}. 

\begin{defn}\label{d: multiplication mod the countable condensation}
For linear orders $M$ and $L$, define an operation $\cdot_\omega$ (the lexicographic product modulo the countable condensation) by
\[
M \cdot_\omega L := \ot (\faktor{ML}{\sim_\omega}).
\]
\end{defn}

The following lemma implies that the linear order $\omega_1$ is a right identity for $\timesc$.

\begin{lem}\label{right absorption for countable condensation}
Let $M$ be any linear order. Then $\faktor{M \omega_1}{\sim_\omega} \cong M$.
\end{lem}

\begin{proof}
Let $M$ be any linear order. Recall that $M \omega_1$ is formed by replacing every $m \in M$ with a copy $\omega_1(m)$ of $\omega_1$. Let $x, y \in M \omega_1$. If $x, y \in \omega_1(m)$ for some $m \in M$ -- that is, if $x$ and $y$ are in the same copy of $\omega_1$ -- then $x \sim_\omega y$, by Lemma \ref{omega_1 mod cc is 1}. This means that for every $x \in M \omega_1$, $\boldsymbol{c}_\omega(x) \supseteq \omega_1(m)$, where $\omega_1(m)$ is the copy of $\omega_1$ containing $x$. 

Now suppose $x$ and $y$ are in different copies of $\omega_1$; without loss of generality, $x \in \omega_1(m)$ and $y \in \omega_1(m')$ for some $m, m' \in M$ with $m<m'$. Then there are at least $\omega_1$-many elements of $M \omega_1$ between $x$ and $y$, namely the $\omega_1$-many elements of $\omega_1(m)$ that are greater than $x$. This means that for any $x \in M \omega_1$, $\boldsymbol{c}_\omega(x) \subseteq \omega_1(m)$ where $\omega_1(m)$ is the copy of $\omega_1$ containing $x$.

For $x \in M \omega_1$, denote by $m_x$ the element of $M$ such that $x \in \omega_1(m_x)$. Then we have $\boldsymbol{c}_\omega(x) < \boldsymbol{c}_\omega(y)$ if and only if $\omega_1(m_x) < \omega_1(m_y)$ if and only if $m_x < m_y$; that is, $\faktor{M \omega_1}{\sim_\omega}$ is isomorphic to $M$.
\end{proof}

It is straightforward to verify that Lemma \ref{right absorption for countable condensation} remains true if $\omega_1$ is replaced by $\omega_1^*$. By contrast, the result fails if we replace $\omega_1$ by any countable linear order $L$: for if $L$ is countable, then for any countable linear order $M \not\cong 1$, $\faktor{ML}{\sim_\omega} \cong 1 \not\cong M$ because $ML$ is countable. In Theorem \ref{TFAE for right identity mod cc}, we characterize the right identities for multiplication modulo the countable condensation: all of those linear orders $L$ such that $\faktor{ML}{\sim_\omega} \cong M$ for every $M$. (One could also say that such $L$ are exactly the linear orders that are always absorbed on the right when multiplying modulo the countable condensation.)

\begin{lem}\label{cf omega 1 iff no ctbl tail}
Let $L$ be a linear order with $\faktor{L}{\sim_\omega} \cong 1$. Then $\cf(L) = \omega_1$ if and only if $L$ has no countable tail.
\end{lem}

\begin{proof}
Let $L$ be a linear order with $\faktor{L}{\sim_\omega} \cong 1$.

Suppose $\cf(L)=\omega_1$, and let $T$ be a tail of $L$: so we can write $L=L'+T$ for some $L'$. Choose $x_0 \in T$ and a strictly increasing, cofinal sequence $\langle x_\alpha: \alpha < \omega_1 \rangle$ in $L$ that begins with $x_0$. Since $T$ is a tail, it is closed upwards, so $x_\alpha \in T$ for each $\alpha < \omega_1$. Thus $T$ is uncountable. Since the arbitrarily chosen tail $T$ was uncountable, $L$ has no countable tail.

Now suppose every tail of $L$ is uncountable. Let $\langle x_\alpha: \alpha < \mu \rangle$, for some ordinal $\mu$, be strictly increasing and cofinal in $L$. Then, in particular, for every $l>x_0$, there is an $\alpha < \mu$ such that $l < x_\alpha$. This means that we can write $L \uparrow x_0=\bigcup_{\alpha < \mu} [x_0,x_\alpha)$. Since $L \uparrow x_0$ is a tail, it is uncountable, by assumption; and since we have $\faktor{L}{\sim_\omega} \cong 1$, it must be that $|[x_0, x_\alpha)| \leq \aleph_0$ for each $\alpha < \mu$. Therefore we have
\[
\omega_1 \leq |L \uparrow x_0| = \left| \bigcup_{\alpha < \mu} [x_0,x_\alpha) \right| \leq \sum_{\alpha < \mu}|[x_0, x_\alpha)| \leq \sum_{\alpha < \mu} \aleph_0 = |\mu| \aleph_0 = \max\{|\mu|, \aleph_0 \}.
\]
Then $\omega_1 \leq \max\{|\mu|, \aleph_0\}$, so $|\mu| \geq \omega_1$. Thus the sequence $\langle x_\alpha: \alpha \leq \mu \rangle$ is uncountable; and so, since this was an arbitrary cofinal sequence, $\cf(L) \geq \omega_1$. We have $\cf(L) \leq \omega_1$ by Corollary \ref{cf at most omega 1 when L condenses to  1}, so $\cf(L) = \omega_1$.
\end{proof}

Reversing the order, we obtain the reverse version of Lemma \ref{cf omega 1 iff no ctbl tail}. 

\begin{lem}\label{reverse omega 1 iff no ctbl head}
Let $L$ be a linear order with $\faktor{L}{\sim_\omega} \cong 1$. Then $\cf^*(L) = \omega_1$ if and only if $L$ has no countable head. \hfill \qed
\end{lem}

\begin{thm}\label{TFAE for right identity mod cc}
Let $L$ be any linear order. The following are equivalent:
\begin{enumerate}
\item $\faktor{ML}{\sim_\omega} \cong M$ for every linear order $M$;
\item $\faktor{L}{\sim_\omega} \cong 1$ and ($\cf(L) = \omega_1$ or $\cf^*(L) = \omega_1$);
\item $\faktor{L}{\sim_\omega} \cong 1$ and ($L$ has no countable tail or $L$ has no countable head);
\item $L$ embeds into the $\omega_1$-lengthened rational line $U$ and $L$ has a strictly monotone sequence of length $\omega_1.$
\item $L$ is isomorphic to an uncountable suborder of $U$.
\end{enumerate}
\end{thm}

\begin{proof}
(1) $\implies$ (5): Suppose $\faktor{ML}{\sim_\omega} \cong M$ for every linear order $M$. In particular, this holds for $M=1$, so $L \preceq U$ by Theorem \ref{main result, universal set}. If $L$ were countable, then we would have $\faktor{\Q L}{\sim_\omega} \cong 1 \not\cong \Q$ because $\Q L$ would also be countable; so it must be that $L$ is uncountable. 

(5) $\implies$ (2): Suppose $L$ is isomorphic to an uncountable suborder of $U$. By Theorem \ref{main result, universal set}, $\faktor{L}{\sim_\omega} \cong 1$; and then by Lemma \ref{uncountable suborders of U have cf at least omega1}, $\cf(L)=\omega_1$ or $\cf^*(L) = \omega_1$.

(2) $\implies$ (3): Suppose $\faktor{L}{\sim_\omega} \cong 1$, and also suppose that $\cf(L)=\omega_1$ or $\cf^*(L)=\omega_1$. If $\cf(L)=\omega_1$, then $L$ has no countable tail by Lemma \ref{cf omega 1 iff no ctbl tail}; and if $\cf^*(L)=\omega_1$, then $L$ has no countable head by Lemma \ref{reverse omega 1 iff no ctbl head}.

(3) $\implies$ (1): Suppose $\faktor{L}{\sim_\omega} \cong 1$, and suppose also that $L$ has no countable tail or $L$ has no countable head. Let $M$ be any linear order, and consider the lexicographic product $ML$. This is the linear order formed by replacing each element $m$ of $M$ by a copy $L_m$ of $L$. For $x \in ML$, denote by $m_x$ the unique element of $M$ for which $x \in L_{m_x}$. To show that $\faktor{ML}{\sim_\omega} \cong M$, it is enough to show that $\cc(x) = L_{x_m}$ for each $x \in ML$, as we have $x \leq y$ in $ML$ if and only if $\cc(x) \leq \cc(y)$ in $\faktor{ML}{\sim_\omega}$. 

Let $x \in LM$. We have $\faktor{L_{m_x}}{\sim_\omega} \cong 1$ by assumption, so $\cc(x) \supseteq L_{m_x}$. Now take $y \in ML$ with $y \not\in L_{m_x}$; without loss of generality, say $y \in L_{m'}$ for some $m' > m_x$. We have assumed that $L$ has no countable tail or no countable head. First suppose $L$ has no countable tail. Then also $L_{m_x}$ has no countable tail, so in particular $\{z \in L_{m_x}: z \geq x\}$ is uncountable. Then as $\{z \in L_{m_x}: z \geq x\} \subseteq [x,y]$, $[x,y]$ is uncountable, and so $x \not\sim_\omega y$. Next suppose $L$ has no countable head. Then, similarly, we have $x \not\sim_\omega y$ because $[x,y]$ contains the uncountable set $\{z \in L_{m'}: z \leq y\}$, so that $\cc(x) \subseteq L_{m_x}.$ Thus $\cc(x) = L_{m_x}$.

(4) $\implies$ (5): If $L$ embeds into $U$ and $L$ has a strictly monotone sequence of length $\omega_1$, then $L$ is uncountable as well.

(5) $\implies$ (4): Suppose $L$ is isomorphic to an uncountable subset of $U$. For purposes of this proof, we may suppose that $L \subseteq U$; moreover, by the pigeonhole principle, either $|L \cap (U \uparrow u_0)| > \aleph_0$ or $|L \cap (U \downarrow u_0)|>\aleph_0$.  Say $|L \cap (U \uparrow u_0)|$ is uncountable. (The argument is similar if we assume instead that $|L \cap (U \downarrow u_0)|$ is uncountable.) Since each interval $\{u_\alpha\} \cup \Q(\alpha)$ of the right half of $U$ is countable, we have -- again by the pigeonhole principle -- that $L \cap (\{u_\alpha\} \cup \Q(\alpha)) \neq \emptyset$ for $\omega_1$-many $\alpha < \omega_1$. Then for each $\alpha < \omega_1$, we can pick an $l_\alpha \in L$ such that $l_\alpha \geq u_\alpha$. These points $l_\alpha$, for $\alpha < \omega_1$, might not all be distinct, but there must be $\omega_1$ distinct elements of $L$ among them, and this gives us the desired uncountable strictly monotone sequence in $L$.
\end{proof}

By Theorem \ref{TFAE for right identity mod cc}, the right identities for multiplication modulo the countable condensation are exactly the uncountable linear orders that embed into the $\omega_1$-lengthened rational line $U$. In the next section, we show that the corresponding set of order types, under $\timesc$, forms a left-regular band.

\section{The left-regular band of right identities for multiplication modulo the countable condensation}\label{s: left-regular band}

For the remainder of the paper, we deal with order types of linear orders; thus, when we write $L \timesc M \cong N$, we mean that whenever $X$ and $Y$ are linear orders with order types $L$ and $M$ respectively, the order type of $X \timesc Y$ is $N$.

\begin{defn}\label{d: semigroup}
A \textbf{semigroup} is a set with an associative binary operation.    
A \textbf{band} is a semigroup in which every element is idempotent. 
A \textbf{left-regular band} is a band $B$ such that $xyx=xy$ for all $x, y \in B$. 
\end{defn}

\begin{thm}\label{big left-regular band}
Denote by $\mathcal{S}$ the set of all order types of linear orders $L$ with the property that $M \timesc L = \faktor{ML}{\sim_\omega} \cong M$ for all linear orders $M$. (That is, $\mathcal{S}$ is the set of right identities for $\timesc$.) Then $\langle \mathcal{S}, \sim_\omega \rangle$ is a left-regular band.
\end{thm}

\begin{proof}
\textit{(Closure)} Let $L_1, L_2 \in \mathcal{S}$ and let $M$ be any linear order. Then 
\begin{align*}
\faktor{(M(L_1 \cdot_\omega L_2))}{\sim_\omega} & \cong \faktor{\left(M \left( \faktor{L_1 L_2}{\sim_\omega} \right) \right)}{\sim_\omega} \\
& \cong \faktor{(M L_1)}{\sim_\omega} \quad \textrm{(as $L_2 \in \mathcal{S}$)} \\
& \cong M \quad \textrm{(as $L_1 \in \mathcal{S}$)}
\end{align*}
Therefore $L_1 \cdot_\omega L_2 \in \mathcal{S}$.

\textit{(Associativity)} Let $L_1, L_2, L_3 \in \mathcal{S}$. Then 
\begin{align*}
(L_1 \cdot_\omega L_2) \cdot_\omega L_3 & \cong \left(\faktor{(L_1 L_2)}{\sim_\omega} \right) \cdot_\omega L_3 \\
& \cong L_1 \cdot_\omega L_3 \quad \textrm{(as $L_2 \in \mathcal{S}$)} \\
& \cong \faktor{(L_1 L_3)}{\sim_\omega}  \\
& \cong L_1 \quad \textrm{(as $L_3 \in \mathcal{S}$)},
\end{align*}

and also

\begin{align*}
L_1 \cdot_\omega (L_2 \cdot_\omega L_3) & \cong L_1 \cdot_\omega \left( \faktor{(L_2 L_3)}{\sim_\omega} \right) \\
& \cong L_1 \cdot_\omega L_2 \quad \textrm{(as $L_3 \in \mathcal{S}$)} \\
& \cong \faktor{(L_1 L_2)}{\sim_\omega} \\
& \cong L_1 \quad \textrm{(as $L_2 \in \mathcal{S})$}.
\end{align*}
Therefore $(L_1 \cdot_\omega L_2) \cdot_\omega L_3 \cong L_1 \cdot_\omega (L_2 \cdot_\omega L_3).$

\textit{(Idempotence)} Let $L \in \mathcal{S}$. Then $L \cdot_\omega L \cong \faktor{LL}{\sim_\omega} \cong L$ by definition of $\mathcal{S}$.

\textit{(Left-regular property)} Let $L_1, L_2 \in \mathcal{S}$. Then 
\begin{align*}
L_1 \cdot_\omega L_2 \cdot_\omega L_1 & \cong (L_1 \cdot_\omega L_2) \cdot_\omega L_1 \quad \textrm{(by associativity)} \\
& \cong \left( \faktor{(L_1 L_2)}{\sim_\omega} \right) \cdot_\omega L_1 \\
& \cong L_1 \cdot_\omega L_1 \quad \textrm{(as $L_2 \in \mathcal{S}$)} \\
& \cong L_1 \quad \textrm{(by idempotence)} \\
& \cong \faktor{(L_1 L_2)}{\sim_\omega} \quad \textrm{(as $L_2 \in \mathcal{S}$)} \\
& \cong L_1 \cdot_\omega L_2.
\end{align*}
Therefore $L_1 \cdot_\omega L_2 \cdot_\omega L_1 \cong L_1 \cdot_\omega L_2$.
\end{proof}

The set of \textit{all} order types of suborders of $U$ (not just the uncountable ones) does not form a left-regular band, for idempotence fails: if $L \preceq U$ is countable, nonempty, and not the one-element linear order, then $LL$ is also countable, so that $L \cdot_\omega L \cong 1 \not\cong L$. However, we do have a weaker algebraic structure. 

\begin{lem}\label{countable second term}
Suppose $L_1$ is any suborder of $U$ and $L_2$ is a countable suborder of $U$. Then $L_1 \timesc L_2 \cong 1$.
\end{lem}

\begin{proof}
Let $L_1$ and $L_2$ be as above. $L_1L_2$ is the linear order obtained by replacing every $l \in L_1$ with a copy $L_2(l)$ of $L_2$. Let $x, y \in L_1L_2$.

\textit{Case 1: } Suppose $x, y \in L_2(l)$ for some $l \in L_1$ (that is, $x$ and $y$ are in the same copy of $L_2$). Since $\faktor{L_2}{\sim_\omega} \cong 1$ by Theorem \ref{main result, universal set}, $x \sim_\omega y$.

\textit{Case 2: } Suppose $x \in L_2(l)$ and $y \in L_2(m)$ for some $l, m \in L_1$ with $l \neq m$; say $l<m$. Then  
\[
[x,y] \subseteq L_2(l) \cup L_2(m) \cup \left( \bigcup_{l < k < m} L_2(k)\right),
\]
and this is a countable union of countable sets: each $L_2(k)$, for $l \leq k \leq m$, is countable because $L_2$ is countable; and $\{ k \in L_1: l < k < m\}$ is countable because $\faktor{L_1}{\sim_\omega} \cong 1$. Therefore $|[x,y]| \leq \omega$, so that $x \sim_\omega y.$

Thus we have that $x \sim_\omega y$ for all $x, y \in L_1L_2$, and so $\faktor{L_1L_2}{\sim_\omega} \cong 1$.
\end{proof}

\begin{thm}\label{suborders of U is a semigroup}
Denote by $\mathcal{X}$ the set of all order types of linear orders that embed into $U$. Then $\langle \mathcal{X}, \sim_\omega \rangle$ is a semigroup.
\end{thm}

\begin{proof}
By Theorem \ref{main result, universal set}, $\mathcal{X}$ consists of the order types of those $L$ for which $\faktor{L}{\sim_\omega} \cong 1$. By Theorem \ref{TFAE for right identity mod cc}, $\mathcal{S}$ consists of the uncountable suborders of $U$, and $\mathcal{X} \setminus \mathcal{S}$
consists of the countable suborders of $U$.

\textit{Closure: } Let $L_1, L_2 \in \mathcal{X}$. Then $L_1 \cdot_\omega L_2$ is isomorphic to a suborder of $L_1$:
\begin{align*}
L_1 \cdot_\omega L_2  = \faktor{L_1 L_2}{\sim_\omega}  & \preceq \faktor{L_1 U}{\sim_\omega} \quad \textrm{(as  $L_2 \preceq U$)}\\
& \cong L_1 \quad \textrm{(by Theorem \ref{TFAE for right identity mod cc})}.
\end{align*}
By assumption, $\faktor{L_1}{\sim_\omega} \cong 1$, so also $\faktor{(L_1 \cdot_\omega L_2)}{\sim_\omega} \cong 1$. Thus $L_1 \cdot_\omega L_2 \in \mathcal{X}$.

\textit{Associativity:} Let $L, M, N \in \mathcal{X}$. There are $8$ cases.

\textit{Case 1: } If $L, M, N \in \mathcal{S}$, then by definition of $\mathcal{S}$,
\[
L \timesc (M \timesc N) \cong L \timesc M \cong L \cong L \timesc M \cong (L \timesc M) \timesc N.
\]
Cases 2 through 7 use both the definition of $\mathcal{S}$ and Lemma \ref{countable second term}.

\textit{Case 2: } Suppose $L \in \mathcal{S}$, $M \in \mathcal{S}$, and 
$N \in \mathcal{X} \setminus \mathcal{S}$. Then
\[
L \timesc (M \timesc N) \cong L \timesc 1 \cong 1 \cong  L \timesc N \cong (L \timesc M) \timesc N. 
\]
\textit{Case 3: } Suppose $L \in \mathcal{S}$, $M \in \mathcal{X} \setminus \mathcal{S}$, and $N \in \mathcal{S}$. Then 
\[
L \timesc (M \timesc N) \cong L \timesc M \cong 1 \cong L \timesc M \cong (L \timesc M) \timesc N.
\]
\textit{Case 4: } Suppose $L \in \mathcal{X} \setminus \mathcal{S}$, $M \in \mathcal{S}$, and $N \in \mathcal{S}$. Then 
\[
L \timesc (M \timesc N) \cong L \timesc M \cong L \cong L \timesc M \cong (L \timesc M) \timesc N.
\]
\textit{Case 5: } Suppose $L \in \mathcal{S}$, $M \in \mathcal{X} \setminus \mathcal{S}$, and $N \in \mathcal{X} \setminus \mathcal{S}$. Then 
\[
L \timesc (M \timesc N) \cong L \timesc 1 \cong 1 \cong 1 \timesc N \cong (L \timesc M) \timesc N.
\]
\textit{Case 6: } Suppose $L \in \mathcal{X} \setminus \mathcal{S}$, $M \in \mathcal{S}$, and $N \in \mathcal{X} \setminus \mathcal{S}$. Then
\[
L \timesc (M \timesc N) \cong L \timesc 1 \cong 1 \cong L \timesc N \cong (L \timesc M) \timesc N.
\]
\textit{Case 7: } Suppose $L \in \mathcal{X} \setminus \mathcal{S}$, $M \in \mathcal{X} \setminus \mathcal{S}$, and $N \in \mathcal{S}$.
Then 
\[
L \timesc (M \timesc N) \cong L \timesc M \cong 1 \cong 1 \timesc N \cong (L \timesc M) \timesc N. 
\]
\textit{Case 8: } If $L, M, N \in \mathcal{X} \setminus \mathcal{S}$, then all three are countable, so by Lemma \ref{countable second term},
\[
L \timesc (M \timesc N) \cong L \timesc 1 \cong 1 \cong 1 \timesc N \cong (L \timesc M) \timesc N. 
\]
Thus in all cases we have $L \timesc (M \timesc N) \cong (L \timesc M) \timesc N$, and so $\timesc$ is an associative operation on $\mathcal{X}$.

Therefore we have that $\langle \mathcal{X}, \timesc \rangle$ is a semigroup.
\end{proof}

The diagram below shows schematically the multiplication table for the suborders of the $\omega_1$-lengthened rational line $U$ under multiplication modulo the countable condensation. The rows on the right-hand part of the table (that is, those entries corresponding to products $L_1 \cdot_\omega L_2$ where the second term $L_2$ is uncountable) are constant at the second term.
\[
\begin{tikzpicture}
\draw [very thick] (0,0) -- (0,9);
\draw [very thick] (-1,8) -- (8,8);
\draw (-1,6) -- (8,6);
\draw (2,9) -- (2,0);

\node at (1,9) {\small{countable}};
\node at (1,8.6) {\small{suborders}};
\node at (1,8.2) {\small{of $U$}};

\node at (-1,7.4) {\small{countable}};
\node at (-1,7) {\small{suborders}};
\node at (-1,6.6) {\small{of $U$}};

\node at (5,8.5) {\small{uncountable suborders of $U$}};

\node at (-1.2,3.4) {\small{uncountable}};
\node at (-1.2,3) {\small{suborders}};
\node at (-1.2,2.6) {\small{of $U$}};

\node at (-.5,8.5) {\Large{$\cdot_\omega$}};
\node at (1,3) {\Large{1}};

\node at (1,7) {\Large{1}};

\node at (5,3.5) {$\faktor{L_1 L_2}{\sim_\omega} \cong L_1$};
\node at (5,3) {this is the table for $\mathcal{S}$};
\node at (5,2.5) {(the left-regular band)};
\draw [decorate, decoration = {bumps, mirror}] (2.3,.3) -- (7.7,.3) -- (7.7,5.7);
\draw [decorate, decoration = bumps] (2.3,.3) -- (2.3,5.7) -- (7.7,5.7);

\node at (5,7) {$\faktor{L_1 L_2}{\sim_\omega} \cong L_1$};
\end{tikzpicture}
\]

\section{Future work}\label{s: future work}

\begin{enumerate}
\item Can we characterize the linear orders that obey left- or right-cancellation laws modulo the countable condensation?
\item Can we generalize our results to the condensation that results from declaring $x \sim y$ when $|[\{x,y\}]| \leq \kappa$, for $\kappa$ a regular cardinal?
\end{enumerate}

\end{document}